\documentclass[11pt,letterpaper]{article}

\usepackage[]{geometry}
\usepackage{graphicx}
\usepackage{amsmath}
\usepackage{amsthm}
\usepackage{underlin}
\usepackage[T1]{fontenc}
\usepackage[bitstream-charter]{mathdesign}

\def\name{Crossed products and MF algebras (W. Li, S. Orfanos)}

\usepackage{sectsty}
\sectionfont{\rmfamily\mdseries\Large}
\subsectionfont{\rmfamily\mdseries\itshape\large}

\theoremstyle{plain}
\newtheorem{theorem}{Theorem}
\newtheorem{lemma}{Lemma}
\newtheorem{claim}{Claim}

\newtheorem{proposition}{Proposition}

\theoremstyle{definition}
\newtheorem*{remark}{Remark}
\newtheorem{definition}{Definition}
\newtheorem{example}{Example}

\begin{document}
\title {Crossed products and MF algebras}
\author {Weihua Li, Stefanos Orfanos}
\maketitle

\begin{abstract}
We prove that the crossed product $\mathcal A\rtimes_{\alpha}G$ of a unital finitely generated MF algebra $\mathcal{A}$ by a discrete finitely generated amenable residually finite group $G$ is an MF algebra, provided that the action $\alpha$ is almost periodic. This generalizes a result of Hadwin and Shen. We also construct two examples of crossed product $C^\ast$-algebras whose BDF $Ext$ semigroups are not groups.
\end{abstract}

\vspace{.2 in}



\section*{Introduction}
The purpose of this note is to generalize two recent results concerning crossed products. The first is:
\begin{theorem}[Hadwin--Shen \cite{4}] Suppose that $\mathcal{A}$ is a finitely generated unital MF algebra and $\alpha$ is a homomorphism from $\mathbb{Z}$ into $Aut(\mathcal{A})$ such that there is a sequence of integers $0\le n_1< n_2<\dotsm $ satisfying \[\lim_{j\to \infty} \|\alpha (n_j)a-a\| =0\] for any $a\in \mathcal{A}$. Then $\mathcal{A}\rtimes_{\alpha} \mathbb{Z}$ is an MF algebra.
\end{theorem}

and the second is:
\begin{theorem}[Orfanos \cite{5}] Let $\mathcal{A}$ be a separable unital quasidiagonal algebra and $G$ a discrete countable amenable residually finite group with a sequence of F\o lner sets $F_n$ and tilings of the form $G=K_nL_n$. Assume $\alpha : G\to Aut (\mathcal{A})$ is a homomorphism such that 
\[\max_{l\in L_n\cap K_nK_n^{-1}F_n} \|\alpha (l)a - a\| \to 0 \mbox{ as } n\to \infty\] for any $a\in \mathcal{A}$. Then $\mathcal{A}\rtimes_{\alpha} G$ is also quasidiagonal.
\end{theorem}
which were both motivated by
\begin{theorem}[Pimsner--Voiculescu \cite{6}] Suppose that $\mathcal{A}$ is a separable unital quasidiagonal algebra and $\alpha$ is a homomorphism from $\mathbb{Z}$ into $Aut(\mathcal{A})$ such that there is a sequence of integers $0\le n_1< n_2<\dotsm $ satisfying \[\lim_{j\to \infty} \|\alpha (n_j)a-a\| =0\] for any $a\in \mathcal{A}$. Then $\mathcal{A}\rtimes_{\alpha} \mathbb{Z}$ is also quasidiagonal.
\end{theorem}

MF algebras are important in their own right but also due to their connection to Voiculescu's topological free entropy dimension for a family of self-adjoint elements $x_1, \ldots, x_n$ in a unital $C^\ast$-algebra $\mathcal{A}$ (\cite{8}). The definition of topological free entropy dimension requires that Voiculescu's norm microstate space of $x_1, \dotsc, x_n$ is ``eventually" nonempty, which is equivalent to saying that the $C^\ast$-subalgebra generated by $x_1, \dotsc, x_n$ in $\mathcal{A}$ is an MF algebra. So it is crucial to determine if a $C^\ast$-algebra is MF, in which case its Voiculescu's topological free entropy dimension is well-defined.\\

In the next section we describe another connection, that between MF algebras and the Brown--Douglas--Fillmore $Ext$ semigroup (introduced in \cite{2}). We will then exhibit two new examples of crossed product $C^\ast$-algebras whose $Ext$ semigroup fails to be a group.
\section*{Background}
We start with a few well-known definitions and facts. 

\begin{definition} A discrete countable group $G$ is \emph{amenable} if there is a sequence of finite sets $\{F_n\}_{n=1}^{\infty}$ (called a \emph{F\o lner sequence}) such that $\displaystyle \lim_{n\to\infty} F_n =G$ and $\displaystyle \lim_{n\to \infty} |F_n\triangle F_n s|/|F_n|= 0$ for any $s\in G$. The group $G$ is \emph{residually finite} if for every $e\ne x\in G$, there is a finite index normal subgroup $L$ of $G$ such that $L\ne xL$. Stated differently, finite index normal subgroups of $G$ separate points in $G$. A \emph{tiling} of $G$ is a decomposition $G=KL$, with $K$ a finite set, so that every $x\in G$ is uniquely written as a product of an element in $K$ and an element in $L$.
\end{definition}

\begin{lemma} Assume $G$ is a discrete countable group. Then $G$ is amenable and residually finite if and only if $G$ has a F\o lner sequence $\{F_n\}_{n=1}^{\infty}$ for which there exists a separating sequence of finite index normal subgroups $L_n$ and a sequence of finite subsets $K_n\supset F_n$ such that $G$ has a tiling of the form $G=K_nL_n$ for all $n\ge 1$.  
\end{lemma}

\begin{theorem} A discrete group $G$ is amenable if and only if $C^{\ast}_r(G) \cong C^{\ast} (G)$. Let $\mathcal{A}$ be a $C^\ast$-algebra. If $G$ is amenable and if there is a homomorphism $\alpha : G \to Aut(\mathcal{A})$, then $\mathcal{A}\rtimes_{\alpha ,r}G \cong \mathcal{A}\rtimes_{\alpha}G$.
\end{theorem}

Quasidiagonal operators were first considered by Halmos.

\begin{definition}
A separable family of operators $\{T_1, T_2,\dotsc \}\subset \mathcal{B}(\mathcal{H})$ is \emph{quasidiagonal} if there exists a sequence of finite rank projections $\{P_n\}_{n=1}^{\infty}$ such that $P_n\to I$ in SOT and $\|P_nT_j-T_jP_n\|\to 0$ for all $j\ge 1$ as $n\to \infty$. A separable C*-algebra is \emph{quasidiagonal} if it has a faithful $\ast$-representation to a quasidiagonal set of operators.
\end{definition}

\begin{theorem}[Rosenberg \cite{7}] Let $G$ be a discrete group. If $C_r^{\ast}(G)$ is a quasidiagonal algebra, then $G$ is amenable.
\end{theorem}

MF algebras were introduced by Blackadar and Kirchberg in \cite{1}.
\begin{definition} A separable $C^\ast$-algebra $\mathcal{A}$ is an \emph{MF algebra} if there is an embedding from $\mathcal{A}$ to \[\prod_{t=1}^{\infty} \mathcal{M}_{N_t}(\mathbb{C}) / \sum_{t=1}^{\infty} \mathcal{M}_{N_t}(\mathbb{C}) \]for positive integers $\{ N_t\}_{t=1}^{\infty}$. If $\mathfrak{A} = \left \{A_t \right\}_{t=1}^{\infty}$ is an element of the above C*-algebra, define its norm by $\|\displaystyle \mathfrak{A}\| = \limsup_{t\to \infty} \|A_t\|_{\mathcal{M}_{N_t}(\mathbb{C})} $.
\end{definition}

\begin{theorem}[Blackadar--Kirchberg \cite{1}] A separable $C^\ast$-algebra $\mathcal{A}$ is MF if and only if every finitely generated $C^\ast$-subalgebra of $\mathcal{A}$ is MF. Subalgebras of MF algebras are also MF. Every quasidiagonal algebra $\mathcal{A}$ is MF and the converse is true if, in addition, $\mathcal{A}$ is nuclear.
\end{theorem}

An exciting result connecting quasidiagonal and MF algebras on one hand, and Brown--Douglas--Fillmore theory of extensions on the other, is the following.

\begin{theorem} Let $\mathcal{A}$ be a unital separable MF algebra. If $\mathcal{A}$ is not quasidiagonal, then $Ext(\mathcal{A})$ fails to be a group.
\end{theorem}

 \begin{example}[Haagerup--Thorbs\o rnsen \cite{3}]
 The reduced group $C^\ast$-algebra of the free group on $n$ generators, namely $C^{\ast}_r (\mathbb{F}_n)$, is an MF algebra but it is not quasidiagonal or nuclear (since $\mathbb{F}_n$ is not amenable). Therefore, $Ext(C^{\ast}_r (\mathbb{F}_n))$ is not a group.
 \end{example}
More examples of this flavor were exhibited in \cite{4}. 

\section*{Preliminary Facts}
Here we state a few facts that will be used extensively in the rest of this note. In what follows, $\mathbb{C} (X_1, \dotsc , X_m)$ will denote the set of all non-commutative complex polynomials in $X_1, \dotsc ,X_m, X_1^{\ast},\dotsc , X_m^{\ast}$. The first result gives equivalent definitions for an MF algebra. Refer to \cite{4} and the references therein for a proof.
\begin{proposition} Suppose $\mathcal{A}$ is a unital $C^\ast$-algebra generated by a family of elements $a_1, \dotsc ,a_m$ in $\mathcal{A}$. Then the following are equivalent:
\begin{enumerate}
\item [\emph{(i)}]$\mathcal{A}$ is an MF algebra.
\item [\emph{(ii)}]For any $\epsilon>0$  and any finite subset $\{ f_1, \dotsc ,f_J\}$ of $\mathbb{C} (X_1, \dotsc , X_m)$, there is a positive integer $N$ and a family of matrices $\{ A_1, \dotsc , A_m\}$ in $\mathcal{M}_N (\mathbb{C})$, such that \[ \max_{1\le j\le J} \left| \|f_j(A_1, \dotsc , A_m)\|_{\mathcal{M}_N (\mathbb{C})} - \|f_j (a_1, \dotsc ,a_m)\|_{\mathcal{A}}\right| < \epsilon.\]
\item [\emph{(iii)}]Suppose $\pi : \mathcal{A} \to \mathcal{B}(\mathcal{H})$ is a faithful $\ast$-representation of $\mathcal{A}$ on an infinite dimensional separable complex Hilbert space $\mathcal{H}$. Then there is a family $\displaystyle \{ [a_1]_n, \dotsc , [a_m]_n\}_{n=1}^{\infty} \subset \mathcal{B}(\mathcal{H})$ such that
  \begin{enumerate}
  \item [\emph{(a)}]For each $n\ge 1$,  $\displaystyle \{ [a_1]_n, \dotsc , [a_m]_n\} \subset \mathcal{B}(\mathcal{H})$ is quasidiagonal;
  \item [\emph{(b)}]$\displaystyle \left\| f\negthickspace\left( [a_1]_n,\dotsc , [a_m]_n\right)\right\|_{\mathcal{B}(\mathcal{H})} \to \left\| f\negthickspace \left( a_1,\dotsc , a_m\right)\right\|_{\mathcal{A}}$ as $n\to \infty$, for any $f\in \mathbb{C}(X_1, \dotsc , X_m)$;
  \item [\emph{(c)}]$[a_i]_n \to \pi (a_i)$ in $\ast$-SOT as $n\to \infty$, for every $1\le i\le m$. 
 \end{enumerate}
\end{enumerate}
\end{proposition}

Let $G$ be a discrete countable amenable residually finite group, equipped with F\o lner sets $F_n$, finite sets $K_n$ and finite index normal subgroups $L_n$ such that $F_n\subset K_n$ and $G=K_nL_n$ is a tiling of $G$ for every $n\ge 1$. Consider the family $\{\xi_{yL_n} : y\in K_n\} \subset \ell^2(G)$, with $\displaystyle \xi_{yL_n}=\sum_{x\in yL_n} \phi_n(x) \delta_x $ and \[\phi_n(x) = \sqrt{\frac{|K_n\cap F_nx|}{|F_n|}}.\]
The following two lemmas can be found in \cite{5}. The second is a consequence of the first.
\begin{lemma} Assume $G$ and $\xi_{yL_n}$ are as above. The following are true:
\begin{enumerate} 
\item [\emph{(i)}]For every $n\ge 1$, $\{\xi_{yL_n} : y\in K_n\}$ is an orthonormal family of vectors.  
\item [\emph{(ii)}]For any $s\in G$, \[ \lambda(s) \xi_{yL_n} = \xi_{syL_n} + \sum_{x\in yL_n} \left(\phi_n (x) - \phi_n(sx) \right)\delta _{sx} \mbox{ , with } \sum_{x\in yL_n} \left|\phi_n (x) - \phi_n(sx) \right|^2 \le \frac{|F_n\triangle F_ns|}{|F_n|}.\]
\end{enumerate}
\end{lemma}
\begin{lemma} Assume $G$ and $\xi_{yL_n}$ are as above. If $P_{yL_n} \xi= \langle \xi ,\xi_{yL_n}\rangle\xi_{yL_n}$ and $P_n=\displaystyle \sum_{y\in K_n}P_{yL_n}$, then
\begin{enumerate}
\item [\emph{(i)}]For every $n\ge 1$, $P_n$ is a self-adjoint projection in $\mathcal{B}(\ell^2(G))$ and rank $P_n =|K_n|$.
\item [\emph{(ii)}]$P_n\to I$ in SOT as $n\to \infty$; and for any $s\in G$, $\displaystyle \|P_n\lambda (s) - \lambda (s)P_n\|^2 \le 4\frac{|F_n\triangle F_ns|}{|F_n|}\to 0$ as $n\to \infty$.
\end{enumerate}
\end{lemma}

\section*{Main Result}
Let $\mathcal{A} = \langle a_1,\dotsc , a_m\rangle $ be a unital finitely generated MF $C^\ast$-algebra. We first establish three claims related to our main result.

For $y\in K_n$, $s\in S=\{s_1, s_2, \dotsc , s_k\}\cup\{ e\}\subset G$ and $1\le i\le m$, consider the elements $ \alpha(y^{-1}s^{-1})a_i \in \mathcal{A}$, the quasidiagonal set $\left\{ \left[ \alpha(y^{-1}s^{-1})a_i\right]_n : 1\le i\le m, s\in S, y\in K_n\right\}\subset \mathcal{B}(\mathcal{H})$ obtained from part~(a) of Proposition~1(iii), and a sequence of finite rank projections $Q_n$ in $\mathcal{B}(\mathcal{H})$ such that $Q_n\to I$ in SOT as $n\to \infty$ and $Q_n$ asymptotically commutes with all elements in the above-mentioned set. For every $1\le i\le m$, positive integer $n$ and $s\in S$, define \[A_i^{(s)} = \sum_{y\in K_n} Q_n \left[ \alpha(y^{-1}s^{-1})a_i\right]_n Q_n \otimes P_{yL_n} \mbox{ , } A_i = A_i^{(e)} \mbox{ , and } U_s = Q_n \otimes P_n\lambda (s)P_n.\]

\begin{claim} For any $\epsilon>0$ and any finite subset $\{ q_1, \dotsc , q_J\}$ of $\mathbb{C} (X_1, \dotsc , X_{(k+1)m})$, there is a positive integer $n$ such that, for $N=$ rank $(Q_n\otimes P_n)$,
\[\max_{1\le j\le J} \left| \left\|q_j \left(A_1, \dotsc , A_1^{(s_k)}, \dotsc , A_m,\dotsc , A_m^{(s_k)}\right)\right\|_{\mathcal{M}_N(\mathbb{C})} \negthickspace\negthickspace - \left\|q_j \left( a_1,\dotsc ,\alpha(s_k^{-1})a_1, \dotsc , a_m, \dotsc , \alpha(s_k^{-1})a_m\right)\right\|_{\mathcal{A}} \right| <\epsilon.\]
\end{claim}

\begin{proof} For every $1\le j\le J$, \[\left\|q_j \left(A_1, \dotsc , A_m^{(s_k)}\right)\right\|_{\mathcal{M}_N(\mathbb{C})}  = \max_{y\in K_n} \left\{\left\|q_j \left(Q_n\left[\alpha(y^{-1})a_1\right]_nQ_n, \dotsc , Q_n\left[\alpha(y^{-1}s_k^{-1})a_m\right]_nQ_n\right)\right\|_{\mathcal{M}_{N/|K_n|}(\mathbb{C})} \right\} \]by Lemma 2(i). Now use the quasidiagonality of $\left\{ \left[ \alpha(y^{-1}s^{-1})a_i\right]_n : 1\le i\le m, s\in S, y\in K_n\right\}$, with $n$ sufficiently large, to obtain, for every $1\le j\le J$, \[\left|\left\|q_j \left(Q_n\left[\alpha(y^{-1})a_1\right]_nQ_n, \dotsc , Q_n\left[\alpha(y^{-1}s_k^{-1})a_m\right]_nQ_n\right)\right\| _{\mathcal{M}_{N/|K_n|}(\mathbb{C})}\right. \mbox{\hspace{2.1 in} }\]
\[ \mbox{\hspace{2.3 in} } - \left. \left\|q_j \left(\left[\alpha(y^{-1})a_1\right]_n, \dotsc , \left[\alpha(y^{-1}s_k^{-1})a_m\right]_n\right)\right\|_{\mathcal{B}(\mathcal{H})} \right| < \frac{\epsilon}{2},\]and part~(b) of Proposition~1(iii) to get \[\left|\left\|q_j \left(\left[\alpha(y^{-1})a_1\right]_n, \dotsc , \left[\alpha(y^{-1}s_k^{-1})a_m\right]_n\right)\right\|_{\mathcal{B}(\mathcal{H})} - \left\|q_j \left(\alpha(y^{-1})a_1, \dotsc , \alpha(y^{-1}s_k^{-1})a_m\right)\right\|_{\mathcal{A}} \right| <\frac{\epsilon}{2}.\]The last norm is equal to $\displaystyle \left\|q_j \left( a_1,\dotsc , \alpha(s_k^{-1})a_m\right)\right\|_{\mathcal{A}}$ since $\alpha (y^{-1})$ is a $\ast$-automorphism.
\end{proof}

\begin{claim} For any $s_1, \dotsc , s_k \in G$, any $\epsilon>0$, and any finite subset $\{ p_1, \dotsc , p_J\}$ of $\mathbb{C} (X_1, \dotsc , X_k)$, there is a positive integer $n$ such that  for $N=$ rank $(Q_n\otimes P_n)$,
\[\max_{1\le j\le J} \left| \left\|p_j \left(U_{s_1}, \dotsc , U_{s_k}\right)\right\|_{\mathcal{M}_N(\mathbb{C})} - \left\|p_j \left(\lambda(s_1), \dotsc ,\lambda(s_k)\right)\right\|_{\mathcal{B}(\ell^2(G))} \right| <\epsilon.\]
\end{claim}
\begin{proof} It follows from Lemma 3(ii) and the definition of the projections $Q_n$. \end{proof}

Let $\epsilon >0$ and $s\in F_n \subset G$. Choose appropriately large positive integer $n$ so that $|F_n\triangle F_ns|<\epsilon^2|F_n|/4$. Assume that for every $1\le i\le m$, the action is almost periodic, in the sense that \[\max_{l\in L_n\cap F_nK_nK_n^{-1}} \|\alpha (l)a_i - a_i\| \to 0 \mbox { as } n\to \infty.\]

\begin{claim} For $A_i^{(s)}$, $A_i$, $U_s$ as above and sufficiently large positive integer $n$, $\displaystyle \left\|U_s^{\ast}A_iU_s - A_i^{(s)}\right\|_{\mathcal{M}_N(\mathbb{C})} < \epsilon$.
\end{claim}

\begin{proof}
Without loss of generality, start with a unit vector $\eta\in Q_n\mathcal{H}$ and compute, for $y\in K_n$,
\[\left\|\left(U_s^{\ast}A_iU_s - A_i^{(s)}\right)\eta \otimes\xi_{yL_n} \right\|^2 = \left\|U_s^{\ast}A_i \eta \otimes P_n\left(\xi_{syL_n} + \sum_{x\in yL_n}(\phi_n(x)-\phi_n(sx))\delta_{sx}\right) - A_i^{(s)} \eta \otimes\xi_{yL_n}\right\|^2\]
\[ \le \left\|U_s^{\ast}A_i \eta \otimes \xi_{syL_n} - A_i^{(s)} \eta \otimes\xi_{yL_n}\right\|^2 + \frac{\epsilon^2}{4} \mbox{  , by Lemma 2(ii). }\]
Let $sy = zl' = lz$ with $z\in K_n$ and $l, l' \in L_n$. Then $l=syz^{-1}\in L_n \cap F_nK_nK_n^{-1}$ and $\xi_{syL_n}$ = $\xi_{zL_n}$, which gives
\[ \left\|U_s^{\ast}A_i \eta \otimes \xi_{zL_n} - A_i^{(s)} \eta \otimes\xi_{yL_n}\right\|^2 = \left\|U_s^{\ast} Q_n\left[ \alpha (z^{-1}) a_i \right]_n \eta \otimes \xi_{zL_n} - Q_n \left[ \alpha (y^{-1}s^{-1}) a_i\right]_n \eta \otimes\xi_{yL_n}\right\|^2 \]
\[ = \left\|Q_n\left[ \alpha (z^{-1}) a_i \right]_n \eta \otimes P_n\left( \xi_{yL_n} + \sum_{x\in yL_n}(\phi_n(sx)-\phi_n(x))\delta_{x}\right) - Q_n \left[ \alpha (y^{-1}s^{-1}) a_i\right]_n \eta \otimes\xi_{yL_n}\right\|^2 \] \[ \le \left\|Q_n\left( \left[ \alpha (z^{-1}) a_i \right]_n - \left[ \alpha (y^{-1}s^{-1}) a_i\right]_n \right) \eta \otimes\xi_{yL_n}\right\|^2+\frac{\epsilon^2}{4} \le \left\| \left[ \alpha (z^{-1}) a_i \right]_n - \left[ \alpha (y^{-1}s^{-1}) a_i\right]_n \right\|^2 + \frac{\epsilon^2}{4}.\]
By part~(b) of Proposition~1(iii), \[ \left\| \left[ \alpha (z^{-1}) a_i \right]_n - \left[ \alpha (y^{-1}s^{-1}) a_i\right]_n \right\|_{\mathcal{B}(\mathcal{H})}^2 < \left\| \alpha (z^{-1}) a_i - \alpha (y^{-1}s^{-1}) a_i \right\|_{\mathcal{A}}^2 + \frac{\epsilon^2}{4}.\]

Finally, $\displaystyle \left\| \alpha (z^{-1}) a_i - \alpha (y^{-1}s^{-1}) a_i \right\|_{\mathcal{A}}^2 = \left\| a_i - \alpha(l) a_i \right\|_{\mathcal{A}}^2 < \frac{\epsilon^2}{4}$ by the almost periodicity of the action. 

Overall, for sufficiently large $n$,  \[\displaystyle \left\|\left(U_s^{\ast}A_iU_s - A_i^{(s)}\right)\eta \otimes\xi _{yL_n}\right\|^2 < \epsilon^2.\]
\end{proof}

We are now ready to state the main result.
\begin{theorem} Let $\mathcal{A}=\langle a_1,\dotsc ,a_m\rangle$ be a unital finitely generated MF algebra and $G =\langle s_1,\dotsc ,s_k\rangle$ a discrete finitely generated amenable residually finite group with a sequence of F\o lner sets $F_n$ and tilings of the form $G=K_nL_n$. Assume $\alpha : G\to Aut (\mathcal{A})$ is a homomorphism such that for every $1\le i\le m$,
\[\max_{l\in L_n\cap F_nK_nK_n^{-1}} \|\alpha (l)a_i - a_i\| \to 0 \mbox{ as } n\to \infty.\] Then $\mathcal{A}\rtimes_{\alpha} G$ is also MF.
\end{theorem}

\begin{proof}
Following the proof in \cite{4}, we will show that $\mathcal{A}\rtimes_{\alpha} G$ is an MF algebra by using Proposition~1. More specifically, we will show that for any $\epsilon>0$  and any finite subset $\{ f_1, \dotsc ,f_J\}$ of $\mathbb{C} (X_1, \dotsc , X_{m+k})$, there is a positive integer $N$ and a family of matrices $\{ A_1, \dotsc , A_m, U_{s_1}, \dotsc , U_{s_k}\}$ in $\mathcal{M}_N (\mathbb{C})$, such that \[ \max_{1\le j\le J} \left| \left\|f_j \left(a_1, \dotsc ,a_m, \lambda(s_1), \dotsc ,\lambda(s_k)\right)\right\|_{\mathcal{A}\rtimes_{\alpha} G} - \left\|f_j\left(A_1, \dotsc , A_m, U_{s_1}, \dotsc , U_{s_k}\right)\right\|_{\mathcal{M}_N (\mathbb{C})}\right| < \epsilon.\]

Let $\{ f_1, \dotsc ,f_J\} \subset \mathbb{C} (X_1, \dotsc , X_{m+k})$, and $ A_1, \dotsc , A_m, U_{s_1}, \dotsc ,U_{s_k}, N$ as in Claims~1-3. We first prove that for every $1\le j\le J$ and sufficiently large $n$, \[\left\| f_j \left(a_1, \dotsc ,a_m, \lambda (s_1), \dotsc , \lambda (s_k)\right) \right\|_{\mathcal{A}\rtimes_{\alpha} G} \ge \left\| f_j \left(A_1, \dotsc ,A_m, U_{s_1}, \dotsc U_{s_k}\right) \right\|_{\mathcal{M}_N(\mathbb{C})}. \]
Consider an enumeration of all polynomials in $\mathbb{C} (X_1, \dotsc , X_{(k+1)m})$ (respectively, in $\mathbb{C}(X_1,\dotsc X_k)$) with rational coefficients. For each $t=1, 2, \dotsc$, we can find a positive integer $n_t$ and $N_t=$ rank $(Q_{n_t}\otimes P_{n_t})$, such that, by Claim~1 (respectively, Claim~2) and for all $1\le j\le t$,  \[ \left| \left\|q_j \left(A_1, \dotsc , A_1^{(s_k)}, \dotsc , A_m,\dotsc , A_m^{(s_k)}\right)\right\|_{\mathcal{M}_{N_t}(\mathbb{C})}  - \left\|q_j \left( a_1,\dotsc ,\alpha(s_k^{-1})a_1, \dotsc , a_m, \dotsc , \alpha(s_k^{-1})a_m\right)\right\|_{\mathcal{A}} \right| <\frac{1}{t}\]
(respectively, \[\left| \left\|p_j \left(U_{s_1}, \dotsc , U_{s_k}\right)\right\|_{\mathcal{M}_{N_t}(\mathbb{C})} - \left\|p_j \left(\lambda(s_1), \dotsc ,\lambda(s_k)\right)\right\|_{C_r^{\star}(G)} \right| <\frac{1}{t}).\]

Let, for every $1\le i\le m$ and $s\in S$, \[\mathfrak{A}_i^{(s)} = \left\{ A_i^{(s)} \right\}_{t=1}^{\infty} \in \prod_{t=1}^{\infty} \mathcal{M}_{N_t}(\mathbb{C}) / \sum_{t=1}^{\infty} \mathcal{M}_{N_t}(\mathbb{C}) \mbox{ , } \mathfrak{A}_i = \mathfrak{A}_i^{(e)} \mbox{ , and } \mathfrak{U}_s = \left\{ U_s \right\}_{t=1}^{\infty} \in \prod_{t=1}^{\infty} \mathcal{M}_{N_t}(\mathbb{C}) / \sum_{t=1}^{\infty} \mathcal{M}_{N_t}(\mathbb{C}),\] 
and $\mathcal{C}$ denote the $C^\ast$-algebra generated by $\left\{ \mathfrak{A}_1, \dotsc , \mathfrak{A}_1^{(s_k)}, \dotsc , \mathfrak{A}_m,\dotsc , \mathfrak{A}_m^{(s_k)}\right\}$. 
Consequently, there are embeddings $\rho_1 :\mathcal{A} \to \mathcal{C}$ and $\rho_2 : C_r^{\ast}(G) \to \mathcal{C}$, given by \[ \rho_1 \left(\alpha (s^{-1})a_i\right) = \mathfrak{A}_i^{(s)} \mbox {  and } \rho_2 \left(\lambda (s)\right) = \mathfrak{U}_s \]
with the property that $(\rho_1,\rho_2)$ is a covariant homomorphism (by Claim~3). Therefore, there exists a $\ast$-homomorphism $\rho :\mathcal {A}\rtimes_{\alpha} G \to \mathcal{C}$, with $\rho \left(\alpha (s^{-1})a_i\right) = \mathfrak{A}_i^{(s)} \mbox { , and } \rho \left(\lambda (s)\right) = \mathfrak{U}_s$. It follows that for all $1\le j\le J$,
\begin{eqnarray*} \left\| f_j \left(a_1, \dotsc ,a_m, \lambda(s_1), \dotsc , \lambda(s_k)\right) \right\|_{\mathcal{A}\rtimes_{\alpha} G} &\ge&  \left\| f_j \left(\mathfrak{A}_1, \dotsc , \mathfrak{A}_m, \mathfrak{U}_{s_1}, \dotsc , \mathfrak{U}_{s_k}\right) \right\|_{\mathcal{C}}\\
& =  &\limsup_{t\to \infty} \left\| f_j \left(A_1, \dotsc ,A_m, U_{s_1}, \dotsc , U_{s_k}\right) \right\|_{\mathcal{M}_{N_t}(\mathbb{C})} .\end{eqnarray*}
It remains to show that for every $1\le j\le J$ and sufficiently large $n$, \[\left\| f_j \left(a_1, \dotsc ,a_m, \lambda(s_1), \dotsc , \lambda(s_k)\right) \right\|_{\mathcal{A}\rtimes_{\alpha} G} \le \left\| f_j \left(A_1, \dotsc ,A_m, U_{s_1}, \dotsc , U_{s_k}\right) \right\|_{\mathcal{M}_N(\mathbb{C})} +\epsilon. \]

There exist a positive integer $D$ and families of monomials $p_j^{(d)} \in \mathbb{C}(X_1,\dotsc , X_k)$ and polynomials $q_j^{(d)} \in \mathbb{C}(X_1,\dotsc , X_{(k+1)m})$ for $1\le d\le D$ and $1\le j\le J$, such that \[f_j \negthickspace \left(a_1, \dotsc , a_m, \lambda(s_1), \dotsc , \lambda(s_k)\right) = \displaystyle \sum_{d=1}^{D} p_j^{(d)} \negthickspace\left(\lambda(s_1), \dotsc , \lambda(s_k)\right) q_j^{(d)} \negthickspace \left(a_1,\dotsc ,\alpha(s_k^{-1})a_1, \dotsc , a_m, \dotsc , \alpha(s_k^{-1})a_m\right)\]

by the covariance relation for crossed products. Similarly, for sufficiently large $n$, we can get \[ f_j \left(A_1, \dotsc ,A_m, U_{s_1}, \dotsc , U_{s_k}\right) = \sum_{d=1}^{D} p_j^{(d)} \left(U_{s_1}, \dotsc , U_{s_k}\right) q_j^{(d)} \left(A_1, \dotsc ,A_m^{(s_k)}\right) + r_j\left(A_1, \dotsc ,A_m^{(s_k)}, U_{s_1}, \dotsc , U_{s_k}\right) \] 
with \[\max_{1\le j\le J} \left\|r_j\left(A_1, \dotsc ,A_m^{(s_k)}, U_{s_1}, \dotsc , U_{s_k}\right) \right\|_{\mathcal{M}_N(\mathbb{C})}< \frac{\epsilon}{3}\] by Claim~3 and repeated use of the approximate covariance relation $A_iU_s = U_s A_i^{(s)} + r(A_i, A_i^{(s)}, U_s)$.

We then have $\displaystyle \left\| f_j \left(a_1, \dotsc ,a_m, \lambda(s_1), \dotsc , \lambda(s_k)\right) \right\|_{\mathcal{A}\rtimes_{\alpha} G} - \left\| f_j \left(A_1, \dotsc ,A_m, U_{s_1}, \dotsc , U_{s_k}\right) \right\|_{\mathcal{M}_N(\mathbb{C})}$
\[< \sum_{d=1}^{D} \left\| p_j^{(d)} \left(\lambda(s_1), \dotsc , \lambda(s_k)\right)\right\|_{C^{\ast}(G)} \left\|q_j^{(d)} \left(a_1,\dotsc , \alpha(s_k^{-1})a_m\right)\right\|_{\mathcal{A}}  \hspace{2 in}\]
\[  \hspace{2.2 in}-  \sum_{d=1}^{D} \left\| p_j^{(d)} \left(U_{s_1}, \dotsc , U_{s_k}\right)\right\|_{\mathcal{M}_N(\mathbb{C})} \left\| q_j^{(d)} \left(A_1, \dotsc ,A_m^{(s_k)}\right)\right\|_{\mathcal{M}_N(\mathbb{C})} +\frac{\epsilon}{3}\]
\[\le \sum_{d=1}^{D} \left\| p_j^{(d)} \left(\lambda(s_1), \dotsc , \lambda(s_k)\right)\right\|_{C^{\ast}(G)} \left( \left\|q_j^{(d)} \left(a_1, \dotsc , \alpha(s_k^{-1})a_m\right)\right\|_{\mathcal{A}} - \left\| q_j^{(d)} \left(A_1, \dotsc ,A_m^{(s_k)}\right)\right\|_{\mathcal{M}_N(\mathbb{C})} \right) \]
\[ + \sum_{d=1}^{D} \left( \left\| p_j^{(d)} \left(\lambda(s_1), \dotsc , \lambda(s_k)\right)\right\|_{C^{\ast}(G)} - \left\| p_j^{(d)} \left(U_{s_1}, \dotsc , U_{s_k}\right)\right\|_{\mathcal{M}_N(\mathbb{C})} \right) \left\| q_j^{(d)} \left(A_1, \dotsc ,A_m^{(s_k)}\right)\right\|_{\mathcal{M}_N(\mathbb{C})} +\frac{\epsilon}{3}.\]

Finally, note that  $\left\| q_j^{(d)} \left(A_1, \dotsc ,A_m^{(s_k)}\right)\right\|_{\mathcal{M}_N(\mathbb{C})} \le \left\| q_j^{(d)} \left(a_1, \dotsc ,\alpha(s_k^{-1})a_m\right)\right\|_{\mathcal{A}} +1$ for $n$ sufficiently large, so define
\[M= \displaystyle \max_{1\le j\le J} \left\{ \sum_{d=1}^{D}\left\| p_j^{(d)} \left(\lambda(s_1), \dotsc , \lambda(s_k)\right)\right\|_{C^{\ast}(G)} , \sum_{d=1}^{D}\left\| q_j^{(d)} \left(a_1, \dotsc ,\alpha(s_k^{-1})a_m\right)\right\|_{\mathcal{A}}+ D\right\}\]
and use Claims~1 and 2 with a larger $n$ if necessary, to obtain, for all $1\le d\le D$ and $1\le j\le J$,
\[ \left\|q_j^{(d)} \left(a_1, \dotsc , \alpha(s_k^{-1})a_m\right)\right\|_{\mathcal{A}} - \left\| q_j^{(d)} \left(A_1, \dotsc ,A_m^{(s_k)}\right)\right\|_{\mathcal{M}_N(\mathbb{C})} < \frac{\epsilon}{3M} \mbox { , and}\]
\[  \left\| p_j^{(d)} \left(\lambda(s_1), \dotsc , \lambda(s_k)\right)\right\|_{C^{\ast}(G)} - \left\| p_j^{(d)} \left(U_{s_1}, \dotsc , U_{s_k}\right)\right\|_{\mathcal{M}_N(\mathbb{C})} < \frac{\epsilon}{3M}. \]
\end{proof}

\begin{remark} Every discrete group is the inductive limit of its finitely generated subgroups. In particular, a discrete maximally almost periodic group is the inductive limit of its residually finite subgroups. Moreover, the inductive limit of MF algebras is an MF algebra. Therefore, Theorem~8 remains true if one assumes $G$ to be any discrete countable amenable residually finite group, and can be extended to all discrete countable amenable maximally almost periodic groups whose finitely generated subgroups satisfy the approximate periodicity condition. One could also assume that $\mathcal{A}$ is separable, rather than finitely generated, because of Theorem~6.
\end{remark}

\section*{Examples}
We may now use this result to construct more exotic examples of crossed product $C^\ast$-algebras whose BDF $Ext$ semigroup is not a group.
\begin{example}
Consider the integer Heisenberg group, which can be defined abstractly as
\[ H =\langle s,t | \left[ [s,t],s\right] , \left[ [s,t],t\right] \rangle.\]
or in a concrete way, as the subgroup of $SL_3(\mathbb{Z} )$ generated by
\[ s= \begin{pmatrix} 1 & 1 & 0\\ 0 & 1 & 0\\ 0 & 0 & 1\end{pmatrix} \mbox{  and } t = \begin{pmatrix} 1 & 0 & 0\\ 0 & 1 & 1\\ 0 & 0 & 1\end{pmatrix} \mbox{ , with } u= [s,t] = s^{-1}t^{-1}st =\begin{pmatrix} 1 & 0 & 1\\ 0 & 1 & 0\\ 0 & 0 & 1\end{pmatrix}. \]
Every element of $H$ can be uniquely written in the form $s^kt^lu^m$, for $k,l.m\in \mathbb{Z}$. The sets \[K_n = \left\{s^kt^lu^m : -\frac{n}{2}< k,l,m\le \frac{n}{2}\right\}\] have subsets $F_n =\left\{s^kt^lu^m : -\sqrt{\frac{n}{2}}< k,l\le \sqrt{\frac{n}{2}}, -\frac{n}{2}<m\le \frac{n}{2}\right\}$ that form a F\o lner sequence, and if we define $L_n = \langle s^{n}, t^{n}, u^{n}\rangle$ then $H=K_nL_n$ is a tiling for every $n\ge 1$. Indeed, $L_n$ is normal since \[(s^kt^lu^m)\left(s^{np}t^{nq}u^{nr}\right)(s^kt^lu^m)^{-1} = s^{np}t^{nq}u^{n(r+qk-pl)} \in L_n\]
and for any $k,l,m \in \mathbb{Z}$, one can find unique $-\frac{n}{2} <k',l',m'\le \frac{n}{2}$ and $p,q,r\in \mathbb{Z}$, such that 
\[ s^kt^lu^m = s^{k'+np}t^{l'+nq}u^{m'+ n(r-pl')} = (s^{k'}t^{l'}u^{m'})\left(s^{np}t^{nq}u^{nr}\right) \in K_nL_n. \]

Assume now that $\mathcal{A}$ is a unital finitely generated non-quasidiagonal MF algebra (e.g. $\mathcal{A} = C^{\ast}_r(\mathbb{F}_2)$) and let an action $\alpha : H \to Aut (\mathcal{A})$ be induced by
\[ \alpha (s) a = \alpha (t) a = e^{2\pi\theta i} a\]
where $a \in \mathcal{A}$ and $0\le \theta\le 1$. Consider a positive integer $n$ with the property that $n\theta$ approximates an integer. Then $\alpha$ is approximately periodic on $L_n\cap F_nK_nK_n^{-1}$. It follows that $(\mathcal{A}, H, \alpha)$ satisfies the conditions of Theorem~8, and thus the crossed product $\mathcal{A}\rtimes_{\alpha} H$ is a non-quasidiagonal MF algebra, therefore its $Ext$ semigroup fails to be a group.
\end{example}
\begin{example}
Consider the Lamplighter group, which can be defined abstractly as
\[ \Lambda = \langle s,t | s^2, [t^jst^{-j}, t^lst^{-l}]: j,l\in \mathbb{Z}\rangle.\]
or otherwise, as the semidirect product $ \left( \bigoplus_{\mathbb{Z}} \mathbb{Z}_2\right) \rtimes \mathbb{Z} $, where the action is by shifting the copies of $\mathbb{Z}_2$ along $\mathbb{Z}$. By denoting $t^jst^{-j} = s_j$, we may write each element of $\Lambda$ uniquely as $s_{j_1}s_{j_2}\dotsm s_{j_k}t^{j_0}$ with $j_1<j_2<\dotsm <j_k$ and $j_0$ in $\mathbb{Z}$. 
Let \[F_n=K_n = \left\{s_{j_1}s_{j_2}\dotsm s_{j_k}t^{j_0} \mbox{ :  }-\frac{n}{2}<j_0\le \frac{n}{2}\mbox{ , }-\frac{n}{2}<j_1<j_2 <\dotsm <j_k\le \frac{n}{2} \mbox{ , and } 0\le k<n\right\}.\]
The F\o lner condition follows immediatety from $F_ns =F_n$ and $|F_n\triangle F_nt|=2|F_n|/n$. Let $L_n$ be the subgroup of $\Lambda$ generated by $t^{n}$ and $s_{j}s_{j+n}$ for all $-\frac{n}{2}<j\le \frac{n}{2}$. Note that $L_n$ contains all elements $s_js_l$ with $l-j$ divisible by $n$, and in fact, $L_n$ is normal in $\Lambda$, since, for $s_{j_1}s_{j_2}\dotsm s_{j_k}t^{j_0} \in \Lambda$ and  $s_{l_1}s_{l_2}\dotsm s_{l_m}t^{l_0} \in L_n$, we have
\[ \left(s_{j_1}s_{j_2}\dotsm s_{j_k}t^{j_0}\right)\left(s_{l_1}s_{l_2}\dotsm s_{l_m}t^{l_0} \right)\left(s_{j_1}s_{j_2}\dotsm s_{j_k}t^{j_0}\right)^{-1} = \left(s_{j_1}s_{j_1+l_0}\right)\dotsm \left(s_{j_k}s_{j_k+l_0}\right)\left( s_{l_1+j_0}s_{l_2+j_0}\dotsm s_{l_m+j_0}t^{l_0}\right)\]
which is in $L_n$. Moreover, $\Lambda=K_nL_n$ is a tiling for every $n\ge 1$, since any $s_{j_1}s_{j_2}\dotsm s_{j_k}t^{j_0} \in \Lambda$ can be decomposed into 
\[\left(s_{r_1}s_{r_2}\dotsm s_{r_m}t^{r_0}\right) \left( s_{r_1 - r_0}\dotsm s_{r_m - r_0} s_{j_1-r_0}\dotsm s_{j_k-r_0}t^{j_0-r_0} \right) \in K_n L_n\] 

with $0\le m\le k< n$, $-\frac{n}{2} <r_0, r_1,\dotsc r_m\le \frac{n}{2}$, $j_0-r_0$ divisible by $n$, and the cardinality of the set $\left( c+n\mathbb{Z} \right) \cap \left\{r_1,\dotsc ,r_m, j_1,\dotsc ,j_k\right\}$ to be an even number (or zero) for each $c=1, \dotsc , n$. To verify the uniquess of such a decomposition, one can easily compute that $K_n^{-1}K_n \cap L_n = \{ e\}$ for all $n\ge 1$.
 
Having studied the group in detail, let us now consider $\mathcal{A}$ to be any unital finitely generated non-quasidiagonal MF algebra and let an action $\alpha : \Lambda \to Aut (\mathcal{A})$ be induced by
\[ \alpha (s) a = a^{\ast} \mbox{ , and } \alpha (t) a = e^{2\pi\theta i} a\]
where $a \in \mathcal{A}$ and $0\le \theta\le 1$. Again, for a positive integer $n$ with the property that $n\theta$ approximates an integer, $\alpha$ is approximately periodic on $L_n\cap K_nK_nK_n^{-1}$. It follows that $(\mathcal{A}, \Lambda, \alpha)$ satisfies the conditions of Theorem~8, and thus the crossed product $\mathcal{A}\rtimes_{\alpha} \Lambda$ is a non-quasidiagonal MF algebra, hence its $Ext$ semigroup is not a group.\end{example}

\end{document}